\documentclass[reqno,12pt]{amsart}
\usepackage{amsmath, latexsym, amsfonts, amssymb, amsthm}

\usepackage{fig4tex}

\numberwithin{equation}{section}

\newtheorem{theorem}{Theorem}[section]
\newtheorem{lemma}[theorem]{Lemma}
\newtheorem{prop}[theorem]{Proposition}

\newtheorem{definition}[theorem]{Definition}
\newtheorem{example}[theorem]{Example}

\newcommand{\gep}{\varepsilon}       

\newcommand{\cF}{{\ensuremath{\mathcal F}} }
\newcommand{\cP}{{\ensuremath{\mathcal P}} }

\newcommand{\bbE}{{\ensuremath{\mathbb E}} }
\newcommand{\E}{{\ensuremath{\mathbb E}} }

\newcommand{\N}{{\ensuremath{\mathbb N}} }

\newcommand{\bbP}{{\ensuremath{\mathbb P}} }

\newcommand{\R}{{\ensuremath{\mathbb R}} }

\newcommand{\bbT}{{\ensuremath{\mathbb T}} }

\newcommand\I{\operatorname{I}}
\newcommand\J{\operatorname{J}}
\renewcommand\H{\operatorname{H}}

\newcommand{\mc}[1]{{\mathcal #1}}

\newcommand{\mb}[1]{{\mathbf #1}}
\newcommand{\bb}[1]{{\mathbb #1}}
\newcommand{\eps}{\varepsilon}

\newfont{\indic}{bbmss12}

\def\un#1{\hbox{{\indic 1}$_{#1}$}}

\title[Large deviations for a random speed particle] {Large deviations
  for a random speed particle}

\author[R.\ Lefevere]{Rapha\"el Lefevere}
 \address{Laboratoire de Probabilit\'es
  et Mod\`eles Al\'eatoires (CNRS UMR 7599), Universit\'e Paris 7
  -- Denis Diderot, UFR Math\'ematiques, Case 7012, B\^atiment
  Chevaleret, 75205 Paris Cedex 13, France}
\email{lefevere\@@math.jussieu.fr}

\author[M.\ Mariani]{Mauro Mariani}
\address{Laboratoire d'Analyse,
Topologie, Probabilit\'es (CNRS UMR 6632), Universit\'e Paul
C\'ezanne Aix-Marseille 3, Facult\'e des Sciences et Techniques
Saint-J\'er\^ome, Avenue Escadqille Normandie-Niemen 13397 Marseille
Cedex 20, France}
\email{mariani\@@cmi.univ-mrs.fr}

 \author[L. Zambotti]{Lorenzo Zambotti}
 \address{Laboratoire de Probabilit{\'e}s
   et Mod\`eles Al\'eatoires (CNRS UMR. 7599)
 Universit\'e Paris 6 -- Pierre et Marie Curie,
U.F.R. Math\'ematiques, Case 188, 4 place
   Jussieu, 75252 Paris cedex 05, France }
 \email{lorenzo.zambotti\@@upmc.fr}

\date{\today}

\begin{document}

\begin{abstract}
  We investigate large deviations for the empirical measure of the position and
  momentum of a particle traveling in a box with hot walls. The particle travels with
  uniform speed from left to right, until it  hits
  the right boundary. Then it is absorbed and re-emitted from the left boundary with
  a new random speed, taken from an i.i.d.\ sequence. It turns out that this simple model, often used to simulate a heat bath,
  displays unusually complex large deviations features, that we explain in detail. In particular,
  if the tail of the update distribution of the speed is sufficiently oscillating, then the empirical 
  measure does not satisfy a large deviations principle, and we exhibit optimal lower and upper large
  deviations functionals.
\end{abstract}



\maketitle

\section{Introduction}
We consider the motion of a particle in a box $[0,1[$. The particle moves with uniform velocity $v_1$ from left to right, until it reaches $1$ and it is instantaneously absorbed and re-emitted at $0$ with a new speed $v_2$. Then the particles travels again through the box with constant speed, and so on. If the sequence $(v_i)_{i \ge 1}$ is i.i.d.\, the stochastic motion we have described is Markovian and arises naturally in the simulation of a heat bath \cite{EckmannYoung, EckmannYoung2,Larralde,LinYoung,Mejia}. In this paper our main goal is to study the large deviations of the law of the empirical measure of the canonical coordinates $(q_t,p_t)$ representing position and momentum of this process.  In spite of the simple description enjoyed by the Markov process $(q_t,p_t)$, it features unusual large deviations properties.

\vspace{5mm}
\figinit{0.3mm}

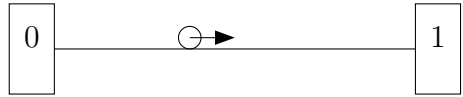
\figpt 1:(-80,0)
\figpt 2:(80,0)

\figpt 3:(-80,20)
\figpt 4:(-80,-20)
\figpt 5:(-100,-20)
\figpt 6:(-100,20)

\figpt 7:(80,20)
\figpt 8:(80,-20)
\figpt 9:(100,-20)
\figpt 10:(100,20)

\figpt 11:(-20,5)
\figpt 12:(0,5)

\figpt 13:(-90,0)
\figpt 14:(90,0)

\psbeginfig{}

\psline[1,2]
\psline[3,4,5,6,3]
\psline[7,8,9,10,7]
\pscirc 11(5)
\psset arrowhead(fillmode=yes) \psarrow[11,12]
\psendfig
\figvisu{\figBoxA}{}
{
\figwriten 13:$0$(1)
\figwriten 14:$1$(1)
\figsetmark{$\figBullet$}
}
\centerline{\box\figBoxA}

\subsection{A non-standard large deviations principle}
A wide literature deals with large deviations of the
empirical measure of Markov processes, after the seminal work of
Donsker and Varadhan \cite{DV}. However, neither their theory or its extensions can be applied in this case, nor they would provide the right result. On one hand, we prove that even the existence of a large deviations principle can fail for certain choices of the marginal law of the i.i.d.\ sequence $(v_i)_i$. On the other hand, when large deviations exist, the associated rate functional can differ from the related Donsker-Varadhan functional. The main point is that, if the random variables $\exp(c/v_i)$ have infinite expectation for some $c>0$, then the probability for the particle to assume a slow velocity of order $t^{-1}$ before time $t$ is not negligible at the large deviations level as $t \to +\infty$. Thus, when looking at events of exponentially small probability, the empirical measure may show features quite far from its typical behavior, and in particular it may concentrate on measures which are singular with respect to the invariant measure of $(q_t,p_t)$ (we recall that this cannot happen if the correct large deviations functional coincides with the Donsker-Varadhan one). 

Another approach to study the large deviations of the empirical measure of the process, would be to use the inversion map for processes depending on an underlying renewal process, see \cite{dm1}. 
However this method is effective only if the sequence $1/v_i$ of times of return to $0$ is bounded, and indeed in the general case one obtains with this approach the wrong rate functional.

In other words, the presence of long tails in the distribution of the return time $1/v_i$ leads standard approaches to fail, and requires a specific analysis. The heavy tails phenomenon induces a slow convergence to the invariant measure (when it exists), and results of the Donsker-Varadhan type are not allowed. It also induces a lack of regularity of the inversion map, and thus renewal techniques cannot be applied directly as well.

\subsection{Setting and notation}
At time $t=0$ the particle is at position $q_0\in[0,1[$ with speed $p_0>0$,
so that the time of the first collision with the wall at $1$ is
\[
T_0=T_0(q_0,p_0):=\frac{1-q_0}{p_0}.
\]
We consider an i.i.d.\ sequence $(v_i)_{i=1,2,\ldots}$ such that $v_i>0$ a.s. for all $i$.
When the particle reaches $1$ for the $i$-th time, it is re-emitted from $0$ with speed $v_i$.
The time to reach $1$ again is then $\tau_i:=1/v_i$.
We denote the law of $\tau_i$ by $\psi(d\tau)$ and the law of $v_i=1/\tau_i$
by $\phi(dv)$.  

Let us then consider the classical delayed renewal process associated with $(\tau_n)_{n\geq 1}$
\[
T_n:=T_0+\tau_1+\cdots+\tau_n, \qquad n\geq 0.
\]
The right-continuous process $(q_t,p_t)_{t\geq 0}$ is now defined by
\begin{equation}
\label{qp}
\begin{split}
(q_t,p_t) & = F(q_0,p_0,t,(\tau_n)_{n\geq 1}) :=
\\
& \quad \begin{cases}
(q_0+p_0t,p_0) & \text{if $t<T_0$,}
\\
\big( \frac{t-T_{n-1}}{\tau_n},\frac 1{\tau_n} \big) &\text{if $T_{n-1}\leq t< T_n$, for some $n\geq 1$.}
\end{cases}
\end{split}
\end{equation}
Next we define the empirical measure of the process $(q_t,p_t)_{t\geq 0}$ as
\[
\mu_t:= \frac 1t \int_{[0,t[}\delta_{(q_s,p_s)}\, ds
\in \mc P([0,1[\times \R_+)
\]
where, for $X$ a metric space, we denote by $\mc P(X)$ the set of Borel probability
measures on $X$, equipped with its narrow (weak) topology.

We first state some basic properties of the process $(q_t,p_t)_{t\geq 0}$ to be proved in section~\ref{mmpp}. Next, we introduce our main results, concerning large deviations principles for the law ${\bf P}_t^{(q_0,p_0)}$ of $\mu_t$, when the set $\mc P([0,1[\times\R_+)$ is equipped with its weak topology.

\subsection{Basic properties}
We define the family of operators 
\[
P_tf(q,p):=\E(f(F(q,p,t,(\tau_{n})_{n}))),
\qquad (q,p)\in [0,1[\times \R_+,
\]
for all bounded Borel function $f:[0,1[\times \R_+\mapsto\R$, where $F$ is defined in \eqref{qp}. The following result is proved in section~\ref{mmpp} below, where Markov properties are more widely discussed.

\begin{prop}
\label{markov}
The process $(q_t,p_t)_{t\geq 0}$ is Markov and $(P_t)_{t\geq 0}$
has the semigroup property: $P_{t+s}=P_tP_s$, $t,s\geq 0$.
\end{prop}

For any probability measure $\mu$ on $\R_+\times[0,1]$ such that
$\mu(p):=\int p\, \mu(dq,dp)\in\R^*_+$ let us set
\[
\tilde \mu(dq,dp):= \frac 1{\mu(p)}\, p\, \mu(dq,dp).
\]
For any $\pi=\pi(dp)\in\cP(\R_+)$ with $\pi(p):=\int p\, \pi(dp)\in\R^*_+$ we also set
\[
\tilde \pi(dp):= \frac 1{\pi(p)}\, p\, \pi(dp).
\]
and we denote
by $\bbP_{\pi}$ the law of an i.i.d.\ sequence $(v_i)_{i\geq 1}$ with
marginal distribution $\pi$, 
i.e.\
\begin{equation}\label{bbqphi}
\bbP_{\pi}:=\otimes_{i\in\N^*}\pi(dv_i).
\end{equation}

\begin{prop}\label{limit}
Let $\pi\in\cP(\R_+)$ with $\pi(p)\in\R^*_+$.
Under $\bbP_{\tilde \pi}$, $\mu_t\rightharpoonup dq\otimes\pi$ a.s. as $t\to+\infty$.
\end{prop}

\subsection{Large deviations rates}
In this section we define the rate functionals $\I$ and $\bar\I$ for the large
deviations of $({\bf P}_t^{(q_0,p_0)})_{t>0}$, and some preliminary notation is needed. First, for convenience of the reader, we recall here the
\begin{definition}
  Let $({\bf P}_t)_{t>0}\subset \mc P([0,1[\times \R_+)$. For 
  two lower semicontinuous functionals $\I,\bar\I:\mc
  P([0,1[\times \R_+) \to [0,+\infty]$, $({\bf P}_t)_{t>0}$ satisfies
\begin{itemize}
\item[-] {a \emph{large deviations upper bound} with 
speed $t$    and 
rate $\I$, if 
\begin{equation}
 \varlimsup_{t \to +\infty} \frac{1}{t} \log {\bf P}_t(\mc C)
  \le - \inf_{u \in \mc C} \I(u)
\end{equation}
for each closed set $\mc C \subset \mc
    P([0,1[\times \R_+)$
}

\item[-] {a \emph{large deviations lower bound} with 
speed $t$    and 
rate $\bar\I$, if
    
\begin{equation}
 \varliminf_{t \to +\infty} \frac{1}{t}  \log {\bf P}_t(\mc O)
  \ge - \inf_{u \in \mc O} \bar\I(u)
\end{equation}
 for each open set $\mc O \subset\mc P([0,1[\times \R_+)$.}
\end{itemize}
The family $({\bf P}_t) _{t>0}$ is said to satisfy a \emph{large deviations
  principle} if both the upper and lower bounds hold with same rate
 $\I=\bar\I$.
\end{definition}

For $X$ a metric space, $\mu \in \mc P(X)$ and $f\in L_1(X,d\mu)$,
$\mu(f) \equiv \mu(f(x)) \equiv \int\!\mu(dx) f(x)$ denotes the
integral of $f$ with respect to $\mu$.
For $\mu,\,\nu$ probability measures on $X$,
$\H(\nu\,|\,\mu)$ is the relative entropy of $\nu$ with respect to $\mu$,
this notation is used regardless of the space $X$.

For $\ell \in [0,1]$ we define the measure $\lambda_\ell \in \mc P([0,1[)$ as
\begin{equation}
\label{e:lambda}
\lambda_\ell(dq):=
\begin{cases}
 \ell ^{-1} \un {[0,\ell[}(q)\, dq
  & \text{if $\ell \in ]0,1]$}
\\
\delta_0(dq)
  & \text{if $\ell = 0$}
\end{cases}
\end{equation}
where $dq$ is the Lebesgue measure on $[0,1[$.
Let us define $\Omega_0\subset\mc P([0,1[\times \R_+)$ as
\begin{equation}
  \label{e:muomega0}
\Omega_0:=\left\{
\mu(dq,dp)= \pi(dp)\, dq, \,
\pi \in \mc P(\R_+^*), \, \pi(p)<+\infty\right\}
\end{equation}
and
$\Omega\subset\mc P([0,1[\times \R_+)$ as
\begin{equation}  \label{e:muomega}
\begin{split}
\Omega:=\Big\{ &
  \mu(dq,dp)= \alpha_1 \pi(dp)\, dq
  +\alpha_2\, \delta_0(dp)\, dq + \alpha_3\, \delta_0(dp) \lambda_\ell(dq) :
  \\ &  \alpha_i\in [0,1], \, \alpha_1+\alpha_2+\alpha_3=1,
  \, \pi \in \mc P(\R_+^*), \, \pi(p)<+\infty, \, \ell\in [0,1) \Big\}
\end{split}
\end{equation}
where here and hereafter we understand $\pi(p):= \int\! p\, \pi(dp) \in ]0,+\infty]$.

If $\mu \in \Omega$ then the writing \eqref{e:muomega} is unique up to
the trivial arbitrary choice of $\pi$ or $\ell$ when respectively
$\alpha_1=0$ or $\alpha_3=0$.
We adopt throughout the paper the convention
\begin{equation}
  \label{convention}
0\cdot \infty =0.
\end{equation}

\begin{definition}
Let
\begin{equation}
\label{e:xiphi}
\xi:= \sup \Big\{c \in \bb R\,:\: \phi(e^{c/p})<+\infty \Big\}
\in[0,+\infty],
\end{equation}
\begin{equation}
\label{e:xiphi2}
\bar\xi:= -\lim_{\delta\downarrow 0}\varliminf_{\gep\downarrow 0}
\gep\log\phi([\gep(1-\delta),\gep(1+\delta)[)\in[0,+\infty].
\end{equation}
If moreover $\pi\in\mc P(\R_+)$ satisfies $\pi(p) <+\infty$,
then we also set
\[
\tilde \pi(dp):= \frac 1{\pi(p)}\, p\, \pi(dp)\in\mc P(\R^*_+).
\]
Then the functionals $\I$ and $\bar\I$ are defined as
\begin{equation}
  \label{e:I2}
\I(\mu):=
  \begin{cases}
\mu(p) \H\big(\tilde\pi\, \big|\, \phi\big)   + (\alpha_2 +\alpha_3 \,\ell^{-1})\, \xi
           & \text{if $\mu \in \Omega$ is given by \eqref{e:muomega} }
\\
+\infty & \text{otherwise}
  \end{cases}
\end{equation}
\begin{equation}
  \label{e:I22}
\bar\I(\mu):=
  \begin{cases}
\mu(p) \H\big(\tilde\pi\, \big|\, \phi\big)   + \alpha_2\, \xi +\alpha_3 \,\ell^{-1}\, \bar\xi
           & \text{if $\mu \in \Omega$ is given by \eqref{e:muomega} }
\\
+\infty & \text{otherwise}
  \end{cases}
\end{equation}
where
\[
\tilde \pi(dp):= \frac 1{\pi(p)}\, p\, \pi(dp)\in\mc P(\R^*_+).
\]
\end{definition}

\begin{lemma}
For any $\phi\in\cP(\R_+^*)$ we have $\bar\xi\geq\xi\geq 0$ and therefore $\bar\I\geq\I$.
\end{lemma}
\begin{proof}
If $c<\xi$ then $\phi(e^{c/p})<+\infty$ and therefore
\[
\phi([\gep(1-\delta),\gep(1+\delta)[) =\int_{[\gep(1-\delta),\gep(1+\delta)[} e^{-c/p}e^{c/p} \, \phi(dp)
\leq e^{-c/(\gep(1-\delta))} \phi(e^{c/p})
\]
so that
\[
\bar\xi= -\lim_{\delta\downarrow 0}\varliminf_{\gep\downarrow 0}
\gep\log\phi([\gep(1-\delta),\gep(1+\delta)[) \geq c
\]
and letting $c\uparrow \xi$ we have the result.
\end{proof}

The following example shows that the inequality $\bar\xi \ge \xi$ can be strict. 
\begin{example} Let 
$\displaystyle{ \phi:= \frac 1Z \sum_{j\geq 0} e^{-2^j}\, \delta_{2^{-j}}}$.
Then  $\xi=1$ and $\overline\xi=+\infty$.
\end{example}
\begin{proof}
For $c\geq 0$
\[
\int e^{c/p}\, \phi(dp) = \frac 1Z \sum_{j\geq 0} e^{(c-1)2^j}
\]
which is finite if and only if $c<1$, so that $\xi=1$. On the other hand, it is easy to check that for $\delta<1/2$ and $\eps_j=3\,2^{-j}$, $\phi([\eps_j(1-\delta),\eps_j(1+\delta)[)=0$, so that $\overline \xi =+\infty$.

\end{proof}

However, for many cases of interest one has $\overline\xi=\xi$. For instance, if $\phi$ is such that
  \begin{equation*}
 \phi([0,p[)= \exp\left(-\frac{\xi+L(p)}{p}\right), \qquad p\geq 0,
  \end{equation*}
  for some function $L$ continuous at $0$, then $\overline\xi=\xi$. This is for instance the case if $\phi(dp)=
  \exp^{-\xi\,p} \xi^{-1} p^{-2} dp$ for some $\xi>0$ (which
  corresponds to exponential interarrival times) or $\phi(dp) =
  p^{\kappa -1} M(p) dp$ for some $\kappa >2$ and some function $M$ slowly varying at $0$ (which yields $\xi=0$ and interarrival times with polynomial decay).

\subsection{Main results}
Recall that a functional $J :  \mc P([0,1[\times \R_+) \to [0,+\infty]$ is \emph{good} if its sublevel sets are compact, namely if the set $\{ \mu \in  \mc P([0,1[\times \R_+) \,:\: J(\mu) \le M\}$ is compact for all $M\ge 0$. In other words, a functional is good iff it is coercive (namely its sublevel sets are precompact) and lower semicontinuous (namely the sublevel sets are closed). 
\begin{prop}
\label{p:igood}
 The functionals $\I$ and $\bar \I$ are good.
\end{prop}
We can now state the main result of this paper.
\begin{theorem}\label{t:ld1}
For all $(q_0,p_0)\in[0,1[\times\,]0,+\infty[$,
the sequence $({\bf P}_t^{(q_0,p_0)})_{t>0}$ satisfies a large deviations
upper bound with with speed $t$ and rate $\I$,
and a large deviations lower bound with with speed $t$ and rate $\bar\I$.

The sequence $({\bf P}_t^{(q_0,p_0)})_{t\geq 0}$ satisfies a large deviations
principle with speed $t$ and good rate $\I$ iff $\xi=\bar\xi$.
\end{theorem}
\subsection{A comparison with previous work}

We note that in \cite{LMZ3} we have studied a large deviations principle for the empirical measure of renewal processes, which turns out to be strictly related to 
Theorem \ref{t:ld1}. We recall the definition of {\it backward recurrence time process} $(A_t)_{t\geq 0}$ and
the {\it forward recurrence time process} $(B_t)_{t\geq 0}$ are defined by
\begin{equation*}
A_t:=t-S_{N_t-1}, \qquad B_t:=S_{N_t}-t, \qquad t\geq 0,
\end{equation*}
where $S_0:=0$, $S_n:=\tau_1+\cdots+\tau_n$, $n\geq 1$, and
\begin{equation*}
N_t:=\sum_{n=0}^\infty \un{(S_n\leq t)}=
\inf\left\{ n\geq 0\,:\: S_{n} >t\right\}
\end{equation*}
is the number of renewals before time $t>0$, see \cite{asmussen}. Then we can see that, in our context, if $T_0=0$ then
\[
q_t = \frac{B_t}{A_t+B_t}, \qquad p_t = \frac1{A_t+B_t}, \qquad t>0.
\]
Therefore one could expect that a LDP for $(A_t,B_t)_{t\geq 0}$ yields an analogous LDP for $(q_t,p_t)_{t\geq 0}$ by a contraction principle.
However, $(q_t,p_t)_{t\geq 0}$ is not a continuous function of $(A_t,B_t)_{t\geq 0}$, in particular if $A_t+B_t\to+\infty$ then $q_t$ can oscillate in $[0,1]$.
Indeed, the LDP for the empirical measure of $(A_t,B_t)_{t\geq 0}$ at speed $t$ holds for any inter-arrival distribution for the i.i.d. sequence $(\tau_i)_{i\geq 1}$ and the   rate functional is similar to $\I$ in \eqref{e:I2}, but it does not contain the last term with $\ell^{-1}\xi$, see \cite{LMZ3}.

In \cite{LefevereZambotti1, LMZ1, LMZ2} our random speed particle in a box is used to construct a class of dynamics which can model the transport of heat in certain materials and which displays anomalous large deviations properties, in particular a lack of analyticity of the LD rate functionals of certain physical observables like the energy current. The results of this paper clarify such anomalies, which are related with the appearance of the additional terms multiplying $\xi$ in the expression \eqref{e:I2} of $\I$.

Finally, we note that the process $(q_t,p_t)_{t\geq 0}$ is a simple example of a {\it piecewise deterministic process}, see \cite{davis, jaco}. For other results on large deviations of a class of piecewise deterministic processes, see \cite{faga}.

\section{The rate functionals}
In this section we study the rate functionals  $\I$ and  $\overline\I$ defined in 
\eqref{e:I2} and, respectively, \eqref{e:I22}.
We recall that any $\mu \in \Omega$ can be written in the form
\eqref{e:muomega}; however $\pi$ or $\ell$ are not uniquely defined if
$\alpha_1=0$ or $\alpha_3=0$ respectively. In order to have a
notational consistency and avoid to distinguish all different
cases, throughout this section we set $\pi=\phi$ whenever
$\alpha_1=0$ and $\ell =1/2$ whenever $\alpha_3=0$.

We denote by $C_\mathrm{b}(\R_+)$ the space of all bounded
continuous functions on $\R_+=[0,+\infty[$.
\begin{lemma}\label{algfact} For all $\pi\in \cP(\R_+)$ and $a>0$
\begin{equation*}
\begin{split}
\pi(p)\, \H\big(\tilde\pi\, \big|\, \phi\big) & =
\sup_{\varphi} \left(\pi(p\varphi) -\pi(p)\log\phi (e^\varphi)\right)
\\ & = \sup_{\varphi\,:\,
\phi (e^\varphi)=a} \left(\pi(p\varphi) -\pi(p)\log\phi (e^\varphi)\right)
= \sup_{\varphi\,:\,
\phi (e^\varphi)=1} \pi(p\varphi)
\end{split}
\end{equation*}
where the suprema are taken over $\varphi\in C_\mathrm{b}(\R_+)$.
\end{lemma}
\begin{proof} It is well known that
\[
\H\big(\tilde\pi\, \big|\, \phi\big) = \sup_{\varphi\in C_\mathrm{b}(\R_+)}
\left(\tilde\pi(\varphi) -\log\phi (e^\varphi)\right).
\]
Now, suppose that $\phi(e^\varphi)=a>0$ and set $\psi:=\varphi-\log a$. Then
\[
\pi(p\varphi) -\pi(p)\log\phi (e^\varphi) = \pi(p\psi) -
\pi(p)\log\phi (e^\psi)
\]
and $\phi(e^\psi)=1$. Therefore the quantity
\[
\sup_{\varphi\,:\,
\phi (e^\varphi)=a} \left(\pi(p\varphi) -\pi(p)\log\phi (e^\varphi)\right)
= \sup_{\varphi\,:\,
\phi (e^\varphi)=1} \left(\pi(p\varphi) -\pi(p)\log\phi (e^\varphi)\right)
\]
does not depend on $a>0$ and is equal to
\[
\begin{split}
 & \sup_a \sup_{\varphi\,:\,
\phi (e^\varphi)=a} \left(\pi(p\varphi) -\pi(p)\log\phi (e^\varphi)\right)
= \sup_{\varphi} \left(\pi(p\varphi) -\pi(p)\log\phi (e^\varphi)\right)
\\ & =
\pi(p)\, \H\big(\tilde\pi\, \big|\, \phi\big).
\end{split}
\]
\end{proof}

The proof of Proposition~\ref{p:igood} is splitted in the two following proofs for $\I$ and $\,\overline{\I}$.
\begin{lemma}\label{lowsem}
The sublevels of $\, \I$ are compact and  $\I$ is lower semicontinuous.
\end{lemma}
\begin{proof}
  Let $(\mu_n)_n\subset\mc P([0,1[\times \R_+)$ be a sequence of probability measures  such that
  \begin{equation}
\label{e:muI}
\varlimsup_{n\to + \infty} \I(\mu_n)
  <+\infty.
\end{equation}
We need to prove that $(\mu_n)_n$ is compact (coercivity of $I$) and
that for any limit point $\mu$ of $(\mu_n)$, $\varliminf_{n\to +\infty}
\I(\mu_n)\ge \I(\mu)$ (lower semicontinuity of $I$). Notice that \eqref{e:muI}
implies $\mu_n \in \Omega$ for $n$ large enough, i.e.\ by
\eqref{e:muomega}, we have
\begin{equation}
\label{e:muomegan}
\mu_n(dp,dq)= \alpha_{1,n} \,
\pi_n(dp) + \alpha_{2,n}\, \delta_0(dp)
dq + \alpha_{3,n}\, \delta_0(dp) \lambda_{\ell_n}(dq)
\end{equation}
with $\pi_n(p)<+\infty$.

\noindent\textit{Coercivity of $\I$.} Let us first show that
\begin{equation}
 \label{e:mup}
\varlimsup_{n\to +\infty} \mu_n(p)<+\infty
\end{equation}
By the bound \eqref{e:muI} and the definition \eqref{e:I2} of $\I$
\begin{equation*}
\mu_n(p) \le \frac{C}{ \H(\tilde\pi_n|\phi)}
\end{equation*}
for some finite constant $C\geq 1$. On the other hand, by
\eqref{e:muomegan}
\begin{equation*}
\mu_n(p) = {\alpha_{1,n}}\,{\pi_n(p)} \le {\pi_n(p)}=\frac1{\tilde\pi_n(p^{-1})}
\end{equation*}
and thus
\begin{equation*}
\varlimsup_n \mu_n(p) \le
\frac{C}{\varliminf_n \H(\tilde\pi_n|\phi) \vee  \tilde\pi_n(p^{-1})}
\end{equation*}
The denominator in the right hand side is uniformly bounded away from
$0$. Indeed, if there is a subsequence $(\tilde\pi_{n_k})_k$ along which
$\H(\tilde\pi_{n_k}|\phi)$ vanishes, then $\tilde\pi_{n_k} \rightharpoonup \phi$, and
thus $\liminf_k \tilde\pi_{n_k}(p^{-1}) \ge \phi(p^{-1})>0$. Thus
\eqref{e:mup} holds.

For each $M>0$, the set $\Omega_M:= \{\mu \in \Omega\,:\: \mu(p)\le
M\}$ is compact, and by \eqref{e:mup}, $\mu_n \in \Omega_M$ for some
$M$ large enough and for any $n$. Thus $(\mu_n)_n$ is compact.

\medskip
\noindent\textit{Semi-continuity of $\I$.}
Let $\mu \in \cP(\R_+)$ be such that along some subsequence, again denoted
$(\mu_n)_n$, $\mu_n \rightharpoonup \mu$.
Passing to subsequences, still
labeled by $n$ for notational simplicity, we can assume that  $\alpha_{i,n} \to \bar{\alpha}_i$ for $i=1,2,3$ and $\ell_n
\to \bar{\ell} \in [0,1]$ as $n\to +\infty$. Note that, in general, one
could have that $\bar{\alpha}_i \neq \alpha_i$ and $\bar{\ell} \neq \ell$.

Since $\mu_n(p)=\alpha_{1,n} \,\pi_n(p)$ is uniformly bounded by \eqref{e:mup}, if
$\bar{\alpha}_1>0$ then $(\pi_n)_n$ is compact in $\mc
P(\R_+)$. Therefore, up to passing to a further subsequence
\begin{equation}
\label{e:psinconv}
\lim_n \alpha_{1,n} \,\pi_n = \bar{\alpha}_1 \big(\beta
\zeta + (1-\beta) \delta_0 \big)
\end{equation}
for some $\beta \in [0,1]$ and $\zeta \in \mc P(\R_+)$
such that $\zeta(\{0\})=0$ and $\zeta(p)<+\infty$. In the same way
\begin{equation*}
  \lim_n \alpha_{3,n} \lambda_{\ell_n}=
   \bar{\alpha}_3 \lambda_{\bar{\ell}}.
\end{equation*}
Thus patching all together, we have $\mu\in\Omega$, in particular
\[
\mu(dp,dq)= \alpha_{1}
\pi(dp)\,dq + \alpha_{2}\, \delta_0(dp)\,
dq + \alpha_{3}\, \delta_0(dp) \,\lambda_{\ell}(dq),
\]
with
\begin{equation}\label{e:allconv}
 \left\{ \begin{array}{ll}
&  \alpha_1 = \bar{\alpha}_1 \beta
\\ \\
&\alpha_2 = \bar{\alpha}_2 +  \bar{\alpha}_1 (1-\beta) + \bar{\alpha_3}
\un {\{1\}}(\bar{\ell})
\\ \\
& \alpha_3 = \bar{\alpha}_3 \un {[0,1[}(\bar{\ell})
\\ \\
& \pi = \zeta \qquad {\rm if }\quad \alpha_1>0
\end{array}
\right.
\end{equation}
and $\pi\in\Omega_0$ is chosen arbitrarily whenever
$\alpha_1=0$. Therefore, we want to prove that
\[
  \begin{split}
\varliminf_{n \to +\infty} \I(\mu_n) & =
\varliminf_{n \to +\infty} \left[\mu_n(p) \H(\tilde\pi_n|\phi)
+ \left(\alpha_{2,n} +\alpha_{3,n} \ell_n^{-1}\right)\xi\,
\right] \\ & \geq
\mu(p) \H(\tilde\pi|\phi) + \left(\alpha_{2} +\alpha_{3} \, \ell^{-1}\right)\xi,
\end{split}
\]
with the usual convention $0\cdot\infty=0$.
Since
\begin{equation*}
 \varliminf_{n \to +\infty} \left(\alpha_{2,n} + \alpha_{3,n} \,
 \ell_n^{-1} \right)
\ge   \bar{\alpha}_2 + \bar{\alpha}_3 \, \bar{\ell}^{-1} =  \alpha_2 +
\alpha_3 \, \ell^{-1} - \bar{\alpha}_1 (1-\beta)
\end{equation*}
to conclude we are left to prove
\begin{equation}
\label{e:tildeH}
\varliminf_{n \to +\infty} \mu_n(p) \H(\tilde\pi_n|\phi) \ge
\bar{\alpha}_1 \beta \, \mu(p) \H(\tilde\pi|\phi) + \xi \,\bar{\alpha}_1 (1-\beta).
\end{equation}

Recall that
\begin{equation*}
\begin{split}
\mu_n(p) \H(\tilde\pi_n|\phi) & = \alpha_{1,n} \,\pi_n(p)
 \sup_{\varphi \in C_\mathrm{b}(\R_+)}
\big(\tilde\pi_n(\varphi) - \log \phi(e^\varphi)\big)
\\
& =  \alpha_{1,n}
 \sup_{\varphi \in C_\mathrm{b}(\R_+)}
 \big(\pi_n(p \varphi) -  \pi_n(p) \, \log \phi(e^\varphi)\big).
\end{split}
\end{equation*}
By a limiting argument, it is easily seen that the supremum in the
above formula can be taken over the set of measurable functions $\varphi$ such that
\begin{equation}
  \label{e:varphi}
  e^\varphi \in  L_1(\R_+, \phi),
\qquad
 p \mapsto p \varphi(p) \in L_1(\R_+, \pi_n).
\end{equation}
Let us fix $c\geq 0$ such that $\phi(e^{c/p})<+\infty$.
For $\delta>0$, let $\chi=\chi_\delta: \R_+ \to [0,1]$ be a smooth
decreasing function such that $\chi(p)=1$ for $p \in
[0,\delta]$ and $\chi(p) =0$ for $p\ge 2\delta$.  For
$\delta>0$ and $\varphi \in
C_\mathrm{b}(\R_+)$, consider the test function
\begin{equation*}
  \varphi_\delta(p) = \frac{c}{p} \, \chi(p) + (1-\chi(p)) \, \varphi(p), \qquad
  p>0.
\end{equation*}
Since $c<\xi$, then
$\varphi_\delta$ satisfies the integrability conditions \eqref{e:varphi} and
therefore
%
\begin{equation}\label{e:bound3}
\begin{split}
 \mu_n(p) \H(\tilde\pi_n|\phi)  \ge & \, \alpha_{1,n}
 \big(\pi_n(p \varphi_\delta) -  \pi_n(p) \, \log \phi(e^{\varphi_\delta})\big)
\\ = & \, \alpha_{1,n} \left[c\,\pi_n(\chi)+\pi_n(p(1-\chi) \varphi)
- \pi_n(p)
 \log \phi \big(e^{\varphi_\delta}\big)\right].
\end{split}
\end{equation}
If now
\begin{equation}
  \label{e:varphi0small}
  \phi(e^{\varphi})<1,
\end{equation}
then for each $c<\xi$ there exists
$\delta_0(c,\varphi)$ such that for each $\delta<\delta_0(c,\varphi)$
\begin{equation}
  \label{e:delta0small}
  \phi(e^{\varphi_\delta})
\le 1.
\end{equation}
In particular, if $\varphi\in
C_\mathrm{b}(\R_+)$ satisfies \eqref{e:varphi0small} and
$\delta <\delta_0(c,\varphi)$, then the logarithm in the last line of
\eqref{e:bound3} is negative. By \eqref{e:psinconv} then
\begin{equation*}
\begin{split}
&  \varliminf_{n \to +\infty} \mu_n(p) \, \H(\tilde\pi_n|\phi)  \ge
\\ & \geq
   \bar{\alpha}_1 (1-\beta) c
+ \bar{\alpha}_1 \beta \, \pi(p (1-\chi) \varphi)
-  \bar{\alpha}_1 \beta \, \pi(p)
 \log \phi(e^{\varphi_\delta}).
\end{split}
\end{equation*}
Since $\pi(\{0\})=\phi(\{0\})=0$ and $e^{\varphi_\delta}\leq
e^{c/p}\in L^1(\phi)$, as $\delta \downarrow 0$
we have $\chi=\chi_\delta\downarrow \un{\{0\}}$ and by
dominated convergence
\[
\pi(p (1-\chi) \varphi)
\to\pi(p \varphi), \qquad
\phi(e^{\varphi_\delta})\to \phi(e^{\varphi}).
 \]
Finally, taking the limit $c\uparrow\xi$, we obtain that
the inequality
\begin{equation}\label{e:bound5}
\begin{split}
\varliminf_n \mu_n(p) \H(\tilde\pi_n|\phi) &\, \ge  \bar{\alpha}_1 (1-\beta) \xi
+ \bar{\alpha}_1 \beta \,\pi(p \varphi)- \bar{\alpha}_1 \beta \, \pi(p)
 \log \phi(e^{\varphi})
\\ & \, =  \bar{\alpha}_1 (1-\beta) \,\xi + \bar{\alpha}_1 \beta \,
\pi(p) \big[\tilde\pi(\varphi) -   \log \phi(e^{\varphi})  \big],
\end{split}
\end{equation}
with the usual convention $0\cdot\infty=0$,
holds for any $\varphi\in C_\mathrm{b}(\R_+)$ satisfying
\eqref{e:varphi0small}. By Lemma~\ref{algfact}, taking the supremum
of the quantity in square brackets in the last line of \eqref{e:bound5}
over all $\varphi$ satisfying \eqref{e:varphi0small}, we obtain
\eqref{e:tildeH}.
\end{proof}

\begin{lemma}\label{lowsemo}
The sublevels of $\, \overline\I$ are compact and  $\overline\I$ is lower semicontinuous.
\end{lemma}
\begin{proof}
Since $\I\leq\overline\I$, then $\left\{\overline\I\leq M\right\}\subseteq\{\I\leq M\}$ and
therefore by Lemma~\ref{lowsem} $\left\{\overline\I\leq M\right\}$ is pre-compact. Let us now show lower semi-continuity. We set $\J:=\overline\I-\I\geq 0$ and we remark that
\[
\J(\mu)=
  \begin{cases}
\alpha_3 \,\ell^{-1}\, \left(\overline\xi-\xi\right)
           & \text{if $\mu \in \Omega$ is given by \eqref{e:muomega} }
\\
+\infty & \text{otherwise}
  \end{cases}
\]
with the usual convention $0\cdot\infty=0$.
Since $\Omega$ is closed in $\mc P([0,1[\times \R_+)$, then $\J$ is lower semi-continuous; indeed, the only non trivial case is for a sequence $\Omega\ni\mu_n\rightharpoonup\mu$ such that $\alpha_{3,n}\to 0$, and in this case by \eqref{e:allconv} above $\mu\in\Omega$ must be given by  \eqref{e:muomega} with $\alpha_3=0$.
Therefore $\J(\mu)=0\leq \varliminf_n \J(\mu_n)$, since $\overline\xi\geq\xi$.
\end{proof}

The next lemma will be used in the following.
\begin{lemma}
  \label{l:Idensity}
Let
\begin{equation}
  \label{e:Omegabar}
  \bar{\Omega}:= \big\{\mu \in \Omega\,:\: \text{ \eqref{e:muomega}
      holds with } \, \alpha_2=0,\, \ell >0 \big\}
\end{equation}
Then $\bar{\Omega}$ is $\overline\I$-dense in $\mc P([0,1[\times \R_+)$,
namely for each $\mu \in \mc P([0,1[\times \R_+)$ such that
$\overline\I(\mu)<+ \infty $ there exists a sequence $(\mu_n)_n$ in $\bar{\Omega}$
such that $\mu_n\rightharpoonup\mu$ and $\varlimsup_n \overline\I(\mu_n) \le \overline\I(\mu)$.
\end{lemma}
\begin{proof}
  Since $\overline\I(\mu)<+\infty$ then $\mu \in \Omega$ and it can be written
  as in \eqref{e:muomega}. If $\alpha_3=1$, then we consider $\ell_n:=\ell+1/n$ and 
  \[
  \mu_n(dq,dp) :=  \delta_0(dp)\, \lambda_{\ell_n}(dq),
  \]
  where $n\in\N$ is large enough to have $\ell+1/n <1$.
  Therefore we can suppose that $\alpha_1+\alpha_2>0$. In what follows, if $\alpha_1=0$ then $\pi:=\gamma$.
  
 Let $c:=\max\{\xi-\gep,0\}$
  for $\gep>0$ and
  define
  \begin{equation*}
    \begin{split}
&    \mu_n(dq,dp) := (\alpha_1+\alpha_2) \, \pi_n(dp)\, dq
 + \alpha_3 \delta_0(dp) \lambda_\ell(dq),
 \end{split}
  \end{equation*}
  where
   \begin{equation*}
    \begin{split}
&    \pi_n(dp) := \frac1{\alpha_1+\alpha_2}\left( 
\alpha_1 \frac{\un{]1/n,+\infty[}(p)\,\pi(dp)}{\pi({]1/n,+\infty[})} + \alpha_2  
\, \frac{e^{c/p}\, \un{[0,1/n]}(p)\,     \gamma(dp)}{\int_{[0,1/n]}e^{c/p}\,   \gamma(dp)} 
 \right).
 \end{split}
  \end{equation*}
Then $\mu_n \in \bar{\Omega}$ and $\mu_n\rightharpoonup\mu$. Note that
\[
\lim_n\H\left( \tilde\pi_n \, | \, \phi \right) = \alpha_1\, \H\left( \tilde\pi \, | \, \phi \right)
+\alpha_2\, c \leq \alpha_1\, \H\left( \tilde\pi \, | \, \phi \right)
+\alpha_2\, \xi
\]
and from this it follows easily that $\varlimsup_n \overline\I(\mu_n) \le \overline\I(\mu)$.
\end{proof}

\section{Preliminary estimates}
\label{preles}
\subsection{Law of large numbers}
In this section we prove Proposition~\ref{limit}, which will come useful in the following.
\begin{proof}[Proof of Proposition~\ref{limit}]
By \eqref{tot}, it is enough to prove that $\bbP_{\tilde \pi}$-a.s. 
$\overline\mu_t\rightharpoonup dq\otimes\pi$ as $t\to+\infty$.
For all $f\in C([0,1]\times\R_+)$ we have
\[
\overline\mu_t(f) =\frac1t \, \sum_{i=1}^{N_t}  \frac1{v_i} \int_0^1f(q,v_i)\, dq
+\frac{1}{t\, v_{N_t+1}}\, \int_0^{v_{N_t+1}(t-S_{N_t})} f(q,v_{N_t+1})\, dq.
\]
By the strong law of large numbers a.s.
\[
\lim_{n\to+\infty} \frac1n\sum_{i=1}^{n}  \frac1{v_i} \int_0^1f(q,v_i)\, dq
=\int_0^1\tilde\pi(f(q,p)/p)\, dq=\frac1{\pi(p)}\int_0^1\pi(f(q,p))\, dq.
\]
By the renewal Theorem, a.s.
\[
\lim_{t\to+\infty} \frac{N_t}t = \frac1{\bbE_{\tilde \pi}(\tau_1)}=\frac{\pi(p)}{\int p\,\frac1p\, \pi(dp)} =\pi(p)\in\R^*_+.
\]
Therefore a.s.
\[
\lim_{t\to+\infty} \frac{{N_t}}t \, \frac1{{N_t}}\sum_{i=1}^{N_t}  \frac1{v_i} \int_0^1f(q,v_i)\, dq=
\int_0^1\pi(f(q,p))\, dq
\]
On the other hand, by the law of large numbers a.s.
\[
\lim_{n\to+\infty} \frac{S_n}n=\tilde\pi(1/p)=\frac1{\pi(p)},
\]
so that
a.s.
\[
\lim_{t\to+\infty} \frac{S_{N_t}}t=\lim_{t\to+\infty} \frac{S_{N_t}}{N_t} \, \frac{{N_t}}t=1,
\qquad \lim_{t\to+\infty} \frac{t-S_{N_t}}t=0.
\]
It follows that a.s.
\[
\lim_{t\to+\infty} \left| \frac{1}{t\, v_{N_t+1}}\, \int_0^{v_{N_t+1}(t-S_{N_t})} f(q,v_{N_t+1})\, dq \right|
\leq \lim_{t\to+\infty} \frac{t-S_{N_t}}t \, \|f\|_\infty =0.
\]
\end{proof}

\subsection{A simplified empirical measure}
We consider the case of a particle which is emitted from $q=0$ at time $t=0$ with initial speed $v_1=1/\tau_1$. In other words, we suppose that $T_0=0$ and we consider the (undelayed) classical renewal process
\[
S_0:=0, \quad S_n:=\tau_1+\cdots+\tau_n, \qquad n\geq 1,
\]
and the corresponding counting process
\[
N_t:=\sum_{n=1}^\infty \un{(S_n\leq t)}, \qquad t\geq 0.
\]
We define the process $(\overline q_t,\overline p_t)_{
t\geq 0}$, where
\[
\begin{split}
(\overline q_t,\overline p_t) & = F(t,(\tau_n)_{n\geq 1}) :=
\left(v_{N_t+1}   \left(t-S_{N_t} \right), v_{N_t+1} \right)
\\ & =
\left( \frac{t-S_{n-1}}{\tau_n},\frac 1{\tau_n}
\right)  \quad {\rm if} \quad
S_{n-1}\leq t< S_n, \qquad n\geq 1, \ t\geq 0.
\end{split}
\]
Then, for any initial condition $(q_0,p_0)\in[0,1[\,\times\,]0,+\infty[$, 
the process $(q_t,p_t)$ can be written in terms of $(\overline q_t,\overline p_t)_{
t\geq 0}$
\begin{equation}\label{qpp}
(q_t,p_t) = \left\{
\begin{array}{ll}
(q_0+p_0t,p_0) \qquad {\rm if} \quad  t<T_0,
\\ \\
(\overline q_{t-T_0},\overline p_{t-T_0}) \qquad {\rm if} \quad  t\geq T_0.
\end{array}
\right.
\end{equation}
We consider now the empirical measure $\overline\mu_t$ of $(\overline q_t,\overline p_t)_{t\geq 0}$
\[
\overline\mu_t:= \frac 1t \int_{[0,t[}\delta_{(\overline q_s,\overline p_s)}\, ds
\in \mc P([0,1[\times \R_+), \qquad t>0,
\]
and we denote by ${\bf P}_t$ the law of $\overline\mu_t$.
An explicit computation shows that for all $f\in C([0,1]\times\R_+)$ we have

%
\begin{equation}\label{1}
\begin{split}
\overline\mu_t(f) & =\frac1t \, \sum_{i=1}^{N_t}  \frac1{v_i} \int_0^1f(q,v_i)\, dq
+\frac{1}{t\, v_{N_t+1}}\, \int_0^{v_{N_t+1}(t-S_{N_t})} f(q,v_{N_t+1})\, dq.
\end{split}
\end{equation}
By \eqref{qpp}, for any initial condition $(p_0,q_0)$ and $t\geq T_0$,
\[
\mu_t=\frac1t\int_0^{T_0} \delta_{(q_0+sp_0,p_0)}\, ds + \frac{t-T_0}t \, 
\overline\mu_{t-T_0},
\]
so that 
\begin{equation}\label{tot}
\|\mu_t-\overline\mu_{t-T_0}\|_{\rm Tot} \leq \frac{2T_0}t, \qquad t\geq T_0,
\end{equation}
where $\|\cdot\|_{\rm Tot}$ denotes the total variation norm.
Therefore, the large deviations rate functionals of $(\mu_t)_{t\geq 0}$ and
$(\overline\mu_t)_{t\geq 0}$ are the same, see \cite[Chap. 4]{demzei}, and 
thus Theorem~\ref{t:ld1} is equivalent to the following
\begin{theorem}\label{t:ld2}
The sequence $({\bf P}_t)_{t>0}$ satisfies a large deviations
upper bound with with speed $(t)$ and rate $\I$
and a large deviations
lower bound with with speed $t$ and rate $\bar\I$.
\end{theorem}

\section{Upper bound at speed $t$}

\subsection{A heuristic argument}

Let us show the basic idea of the upper bound.
Let us suppose for simplicity that $\xi=0$.
We want to estimate from above for $\mc A$ a measurable subset of $\mc P([0,1[ \times \R_+)$
\[
\frac1t \log{\bf P}_t(\mc A) = \frac1t \log\E_\phi\left(
\un{(\overline\mu_t\in \mc A)} \, e^{-t\overline\mu_t(f)}\, e^{t\overline\mu_t(f)}\right)
\leq -\inf_{\pi\in \mc A} \pi(f)+\frac1t \log\E_\phi\left(e^{t\overline\mu_t(f)}\right)
\]
where we choose an arbitrary $f:\R_+\mapsto\R$ such that $\phi(e^{f/v})=1$.
For such $f$ one can see that
\[
\frac1t \log\E_\phi\left(e^{t\overline\mu_t(f)}\right)=0
\]and then we obtain
\[
\limsup_{t\to+\infty}\frac1t \log\bbP(\overline\mu_t\in \mc A)
\leq -\sup_{f: \, \phi(e^{f/p})=1} \left[\inf_{\pi\in \mc A} \pi(f)\right].
\]
By a minimax argument
\[
\sup_{f:\,\phi(e^{f/p})=1}\left[\inf_{\pi\in \mc A}\pi(f)\right]
=\inf_{\pi\in \mc A}\left[\sup_{f:\,\phi(e^{f/p})=1}\pi(f)\right]
=\inf_{\pi\in \mc A}\left[\sup_{\varphi:\,\phi(e^{\varphi})=1}\pi(p\varphi)\right].
\]
By Lemma~\ref{algfact}, the quantity in square brackets is equal to
$\pi(p)  \H(\tilde\pi \, |\, \phi)$ and we have the
upper bound.

\subsection{Exponential tightness}

\begin{lemma}
\label{l:3.2}
\begin{equation}
\label{e:onemom}
\lim_{M\to +\infty} \varlimsup_{t\to +\infty}
\frac{1}{t}  \log \bbP (\overline\mu_t(p)>M)=-\infty.
\end{equation}
In particular the sequence $({\bf P}_t)_{t>0}$ is exponentially tight
with speed $t$, namely
\begin{equation*}
\inf_{\mc K \subset \subset \mc P([0,1[\times \R_+)} \varlimsup_{t\to+\infty}
\frac{1}{t}  \log {\bf P}_t (\mc K)=-\infty.
\end{equation*}
\end{lemma}
\begin{proof}
Note that by \eqref{1}, if $\lfloor tM\rfloor  \ge 1$
\begin{equation*}
\begin{split}
\{\overline\mu_t(p) > M\}
= \left\{\frac{N_t}t + \frac{t-S_{N_t}}{t\,\tau_{N_t+1}} > M \right\}
\subset \big\{N_t+1 > \lfloor t M\rfloor \big\}=
\left\{S_{\lfloor tM\rfloor }  \leq t \right\}
\end{split}
\end{equation*}
Therefore by  the Markov inequality
\begin{equation*}
\begin{split}
\bbP(\overline\mu_t(p) \ge M) \le
 \bbP\left(S_{\lfloor tM\rfloor } \le t\right)
\le e^t \, \E\left( e^{-S_{\lfloor tM\rfloor }}\right)
=  e^{t - \lfloor tM\rfloor  \log c}
\end{split}
\end{equation*}
where $c^{-1}:= \E\left( e^{-1/{v_1}} \right)<1$. The
inequality \eqref{e:onemom} follows by taking the $\limsup$ as $t\to
+\infty$ and $M\to +\infty$. Since for any $M>0$ the set $\{\mu \in
\mc P([0,1[\times \R_+)\,:\: \mu(p) \le M \}$ is compact,
exponential tightness follows.
\end{proof}

\subsection{Free energy}

For $f\in C_c(\bbT\times\R_+)$ we set
\[
\bar f(q,v):=\frac1q \int_0^q f(r,v)\,dr, \qquad (q,v)\in\,]0,1]\times\R_+.
\]
Let $\Lambda$ be the set of all bounded lower semicontinuous functions $f:[0,1[\times \R_+\mapsto\R$ 
such that
\begin{equation}\label{calpha}
C_f:=
\int_{]0,+\infty[} \phi(dv)\, e^{\bar f(1,v)/v} < 1
\end{equation}
and
\begin{equation}
\label{e:fsmall}
D_f:=\sup_{s>0} \int_{(0,1/s]} \phi(dv)\,e^{s\bar f(sv,v)}   <+\infty.
\end{equation}
\begin{prop}\label{mgf} For all $f\in\Lambda$
\begin{equation}\label{freeener}
\sup_{t>0} \, \E_\phi\left(e^{t\overline\mu_t(f)}\right) =
\sup_{t>0} \, \E_\phi\left(\exp\left(\int_0^tf(\overline q_s,\overline p_s)\, ds\right)\right)\leq \frac{D_f}{1-C_f} <+\infty.
\end{equation}
\end{prop}
\begin{proof} Since $C_f>0$, we can introduce the
probability measure
\[
\phi_f(dv):=\frac1{C_f}\,
\phi(dv) \, e^{\frac{\bar f(1,v)}v}
\]
and denote by $\zeta_n$ the law of
$S_n$ under $\bbP_{\phi_f}$.
Then we can write by \eqref{1}
\[
\begin{split}
& \E_\phi\left(\exp\left(\int_0^tf(\overline q_s,\overline p_s)\, ds\right)\right) =
\E_\phi\left(\un{(N_t=0)}\exp\left( t\bar f({tv_1},v_{1})\right) \right)
\\ & +
\sum_{n=1}^{\infty}
\E_\phi\left(\un{(N_t=n)}\exp\left(\sum_{i=1}^n \frac{\bar f(1,v_i)}{v_i}\,
+ (t-S_{n}) \, \bar f((t-S_{n})v_{n+1},v_{n+1})\right)\right)
\\ & = \int_0^{\frac1t} \phi(dv)\, e^{t \, \bar f({tv},v)}
+\sum_{n=1}^{\infty}\int_0^t C_f^n\, \zeta_n(ds) \int_{]0,1/(t-s)[} \phi(dv) \, e^{(t-s) \,
\bar f((t-s)v,v)}
\\ & \leq D_f \sum_{n=0}^{\infty} C_f^n = \frac{D_f}{1-C_f}.
\end{split}
\]
\end{proof}

\begin{lemma}\label{l:3.4}
For all $\mu \in \Omega$ 
  \begin{equation} \label{e:muf}
    \sup_{f\in\Lambda} \mu(f)      \ge \I(\mu).
  \end{equation}
\end{lemma}
\begin{proof}
  For $\varphi \in C_\mathrm{c}(\R_+)$, $c<\xi$,
  $\delta>0$ and $\,m \in (0,1)$, let
 \begin{equation*}
   f_{c,\delta,\varphi,m}(r,p):= p \varphi(p) + c \un {[0,\delta[}(p) \frac{\un {[0,m[}(r)}m, \qquad (r,p)\in\bbT\times\R_+.
 \end{equation*}
Then
\[
 \frac 1p\, \bar f_{c,\delta,\varphi,m}(1,p) =  \varphi(p) + \frac{c}{p} \un {[0,\delta[}(p), \qquad p\in\R_+^*.
\]
Notice that
\[
  \begin{split}
&   \int_{(0,1/s]} \phi(dv)\,e^{s\bar f_{c,\delta,\varphi,m}(sv,v)} = \int_{(0,1/s]} \phi(dp)\,\exp
    \Big(\int_{[0,ps)}\!dr\,\frac{f_{c,\delta,\varphi,m}(p,r)}{p}\Big)
\\ & = \int_{(0,1/s]}
    \phi(dp)\,\exp \Big( ps\,\varphi(p) + \frac{c}{p} \un
    {[0,\delta[}(p) \frac{(ps)\wedge m }{m} \Big)
    \le  e^{\|\varphi\|_{\infty}} \int_{(0,+\infty)} \phi(dp)\, e^{1 \vee
(c/p)}
 \end{split}
\]
which is bounded uniformly in $s$, so that \eqref{e:fsmall} holds for
$f=f_{c,\delta,\varphi,m}$. Let now $a<1$. If
\begin{equation}\label{e:varphismall}
\phi(e^\varphi) =a<1
\end{equation}
then there exists $\delta_0=\delta_0(c,\varphi)$ such that for all $\delta<\delta_0$
\[
C_{f_{c,\delta,\varphi,m}} = \phi\left( e^{\varphi + \frac{c}{p} \un {[0,\delta[}}\right)<1
\]
and therefore $f_{c,\delta,\varphi,m}\in\Lambda$.
Now, if $\mu$ is given by \eqref{e:muomega} then
\[
\mu(f_{c,\delta,\varphi,m}) =
 {\alpha_1}\pi(p\varphi) + c\,\alpha_2 +
 c\,\alpha_3 \,\min\left\{\frac1{\ell }, \frac1{m}\right\}.
 \]
 Since $\pi(p\varphi)=\mu(p) \,\tilde\pi(\varphi)$ then
 \begin{equation*}
 \begin{split}
      \sup_{f\in\Lambda} \mu(f) &  \geq
    \sup_{\varphi} \sup_{c,m}\sup_{\delta} \mu(f_{c,\delta,\varphi,m})
    \\
 & = \mu(p)\, \sup_{\varphi} \left\{\tilde\pi(\varphi)-\log\phi(e^\varphi)\right\} + \mu(p)\, \log a+(\alpha_2+\alpha_3 \ell^{-1})
\, \xi
\end{split}
\end{equation*}
(with the usual convention $0\cdot\infty=0$)
where the supremum on $\delta$ is performed
over $[0,\delta_0(c,\varphi)[$, the supremum on $(c,m)$ over $[0,\xi[\, \times\,
[0,1[$ and the supremum on $\varphi$ over $\varphi \in C_\mathrm{c}(\R_+)$
satisfying \eqref{e:varphismall}. By Lemma~\ref{algfact} the supremum over $\varphi$ satisfying
\eqref{e:varphismall} does not depend on $a$ and equals $\mu(p)\H(\tilde\pi|\phi)$, so that
finally
\[
\sup_{f\in\Lambda} \mu(f)   \geq \sup_{a<1} \left\{ \mu(p)\H(\tilde\pi|\phi)+(\alpha_2+\alpha_3 \ell^{-1})\xi
 +\mu(p)\, \log a\right\} = \I(\mu).
\]
\end{proof}

\begin{proof}[Proof of Theorem~\ref{t:ld1}, upper bound]
For $M>0$, let $\Omega_M=\big\{ \mu \in \Omega\,:\: \mu(p)\le M\big\}$
and
\begin{equation*}
  R_M:= -\varlimsup_{t \to +\infty} \frac{1}{t} \log {\bf P}_t(\Omega_M^c).
\end{equation*}
For $\mc A$ measurable subset of $\mc P([0,1[ \times \R_+)$ and
for $f \in \Lambda$, by \eqref{freeener},
\[
\begin{split}
    \frac{1}{t} \log {\bf P}_t(\mc A) & =
      \frac{1}{t} \log
        \E_\phi\left( e^{t \overline\mu_t(f)}
      e^{-t \overline\mu_t(f)} \un {\mc A} \right)
\le \frac{1}{t} \log\left[   e^{ -t\inf_{\mu \in \mc A}\mu(f)} \,
\E_\phi\left( e^{t \overline\mu_t(f)} \right) \right]
\\ & \leq -\inf_{\mu \in \mc A} \mu(f) +\frac{1}{t} \log\frac{D_f}{1-C_f}
    \end{split}
\]
and therefore
  \begin{equation}\label{e:ldmeas}
      \varlimsup_{t \to + \infty} \frac{1}{t} \log {\bf P}_t(\mc A)
\le  -\inf_{\mu \in \mc A} \mu(f).
\end{equation}
Let now $\mc O$ be
an open subset of $\mc P([0,1[ \times \R_+)$. Then applying
\eqref{e:ldmeas} for $\mc A=\mc O \cap \Omega_M$
  \begin{equation*}
    \begin{split}
&      \varlimsup_{t\to+\infty} \frac{1}{t} \log {\bf P}_t(\mc O) \le
      \varlimsup_{t\to+\infty} \frac{1}{t} \log\big[2 \max({\bf P}_t
      (\mc O \cap \Omega_M),{\bf P}_t(\Omega_M^c) )\big]
\\ &
\le \max\left(-\inf_{\mu \in \mc O \cap \Omega_M} \mu(f), -R_M\right)=
-\inf_{\mu \in \mc O \cap \Omega_M} \mu(f) \wedge R_M
    \end{split}
\end{equation*}
which can be restated as
\begin{equation}\label{e:ldopen}
  \varlimsup_{t\to +\infty}  \frac{1}{t} \log {\bf P}_t(\mc O) \le
- \inf_{\mu \in \mc O} \I_{f,M}(\mu)
\end{equation}
for any open set $\mc O$ and any $f \in \Lambda$ and $M>0$,
where the functional $\I_{f,M}$ is defined as
\begin{equation*}
  \I_{f,M}(\mu):=
  \begin{cases}
    \mu(f) \wedge R_M & \text{if $\mu \in \Omega_M$}
\\ +\infty & \text{otherwise}
  \end{cases}
\end{equation*}
Since $f$ is lower semicontinuous and $\Omega_M$ compact, $I_{f,M}$ is lower
semicontinuous. By minimizing \eqref{e:ldopen} over $f\in \Lambda$ and $M>0$ we obtain
for every open set $\mc O$
\[
\varlimsup_{t\to +\infty}  \frac{1}{t} \log {\bf P}_t(\mc O) \le
- \sup_{f,M}\inf_{\mu \in \mc O} \I_{f,M}(\mu)
\]
and by applying the minimax lemma
\cite[Appendix 2.3, Lemma 3.3]{KL}, we get that for every compact set $\mc K$
\[
\varlimsup_{t\to +\infty}  \frac{1}{t} \log {\bf P}_t(\mc K) \le
- \inf_{\mu \in \mc O} \sup_{f,M} \I_{f,M}(\mu)
\]
i.e.\ $({\bf P}_t)_{t\geq 0}$ satisfies
a large deviations upper bound on compact sets with speed $t$ and rate
\begin{equation*}
\tilde{\I}(\mu):=  \sup \{\I_{f,M}(\mu),\,M>0,\,f \in \Lambda\}, \qquad
\mu\in\cP([0,1[\,\times[0,+\infty[),
\end{equation*}
and since $\lim_{M\to +\infty} R_M=+\infty$ by Lemma~\ref{l:3.2}
\begin{equation*}
\tilde{\I}(\mu)\ge  \sup \{\I_{f}(\mu),\,f \in \Lambda\}
\end{equation*}
where
\begin{equation*}
  \I_{f}(\mu):=
  \begin{cases}
    \mu(f) & \text{if $\mu \in \Omega$}
\\ \\
+\infty & \text{otherwise}
  \end{cases}
\end{equation*}
Thus $\tilde{\I}(\mu) \ge \I(\mu)$ by Lemma~\ref{l:3.4}. Therefore
$({\bf P}_t)_{t\geq 0}$ satisfies a large deviations upper bound with rate $\I$ on
compact sets, and by the exponential tightness proved in
Lemma~\ref{l:3.2}, it satisfies the full large deviations upper bound on closed sets.
\end{proof}

\section{Lower bound at speed $t$}

We are going to prove the lower bound in Theorem~\ref{t:ld1}.

\begin{prop}
For every $\mu\in\Omega$ there exists a family ${\bf Q}_t$ of
probability measures on $\mc P([0,1[\times \R_+)$ such that
${\bf Q}_t\rightharpoonup\delta_\mu$
and
\[
\varlimsup_{t\to+\infty} \frac1t  \H({\bf Q}_t \,| \, {\bf P}_t)\leq \bar\I(\mu).
\]
\end{prop}
\begin{proof} Let us first suppose that $\mu\in\bar\Omega$ as in \eqref{e:Omegabar}, i.e.\
\begin{equation}\label{e:1.1.ug}
\mu(dq,dp)=\alpha\, \pi(dp)\, dq +
(1-\alpha)\, \lambda_\ell(dq)\, \delta_0(dp)
\end{equation}
with $\alpha\in[0,1]$ and $\pi\in\cP(\R_+)$, $m:=\pi(p)\in\R^*_+$, $\ell\in\,]0,1[$.
Notice that $\mu(p)=\alpha\,\pi(p)\in\R_+$.

Suppose first that $\mu(p)=0$, i.e.\ $\alpha=0$. In this case, we define by 
$\bbP^{t,\delta}$ the law on $\R_+^{\N^*}$ such that under
$\bbP^{t,\delta}$ the sequence $(v_i)_{i\geq 1}$ is independent and
\begin{enumerate}
\item $v_{1}$ has distribution
\[
\bbP^{t,\delta}\left(v_{1}\in dv \right) =
\phi\left(dv\, \big|\, K_t\right), \qquad K_t:= \left[\frac{\ell(1-\delta)}{ t},
  \frac{\ell(1+\delta)}{ t} \right[ 
\]
\item for all $i\geq 2$, $v_i$ has law $\phi$.
\end{enumerate}
Under $\bbP^{t,\delta}$, $\overline\mu_t$ is a.s. equal to
\[
\overline\mu_t=\frac{\un{[0,tv_{1}]}(dq)}{tv_{1}}\, \delta_{v_{1}}(dp).
\]
Let us set ${\bf Q}_{t}:=\bbP^{t,\delta}\circ\overline\mu_t^{-1}$. Then we have
\begin{equation*}
\lim_{\delta\downarrow 0}\lim_{t\uparrow +\infty} {\bf Q}_{t} = \delta_\mu.
\end{equation*}
Moreover
\[
\H({\bf Q}_t \,| \, {\bf P}_t) \leq \H\left(\bbP^{t,\delta} \,| \, \bbP_\phi\right)=-\log\phi(K_t),
\]
so that
\[
\varlimsup_{t\to+\infty} \frac1t \H({\bf Q}_t \,| \, {\bf P}_t) \leq - \lim_{\delta \downarrow 0} \varliminf_{\eps \downarrow 0}
\frac\eps\ell\, \log \phi( [\eps(1-\delta),\eps(1-\delta)) ) = \ell^{-1}\, {\overline\xi}=\overline\I(\mu).
\]

We suppose now that $\alpha\in\,]0,1[$ and therefore $\mu(p)=\alpha\pi(p)>0$.
We set $T_t:= \lfloor\mu(p)\, t\rfloor$ and we suppose that
$t\geq 1/\mu(p)$. Let us also fix $\delta\in\,]0,(1-\alpha)/2[$.
Let us define by $\bbP^{t,\delta}$ the law on $\R_+^{\N^*}$ such that under
$\bbP^{t,\delta}$ the sequence $(v_i)_{i\geq 1}$ is independent and
\begin{enumerate}
\item for all $i\leq T_t$, $v_i$ has law $\tilde\pi$
\item $v_{T_t+1}$ has distribution
\[
\bbP^{t,\delta}\left(v_{T_t+1}\in dv \right) =
\phi\left(dv\, \big|\, K_t\right), \qquad K_t:=\Big[\frac{\ell-\delta}{(1-\alpha-\delta) t},
  \frac{\ell+\delta}{(1-\alpha+\delta) t} \Big[ 
\]
\item for all $i\geq T_t+2$, $v_i$ has law $\phi$.
\end{enumerate}
Let us set ${\bf Q}_{t}:=\bbP^{t,\delta}\circ\overline\mu_t^{-1}$. Let us prove now that
\begin{equation}\label{sieg}
\lim_{\delta\downarrow 0}\lim_{t\uparrow +\infty} {\bf Q}_{t} = \delta_\mu.
\end{equation}
By the law of large numbers, under $\bbP_{\tilde\pi}$ we have a.s.
\[
\lim_{t\to+\infty} \frac{S_{T_t}}t =
\lim_{t\to+\infty} \frac{S_{T_t}}{T_t} \,\frac{T_t}t=
\frac1{\mu(p)}\, {\alpha\mu(p)}=\alpha<1.
\]
However $S_{T_t}$
has the same law under $\bbP_{\tilde\pi}$ and under $\bbP^{t,\delta}$, so we obtain
\begin{equation}\label{mire}
\lim_{t\to+\infty} \bbP^{t,\delta}\left(\left|\frac{S_{T_t}}t-\alpha\right|\geq \delta \right)=
\lim_{t\to+\infty} \bbP_{\tilde\pi}\left(\left|\frac{S_{T_t}}t-\alpha\right|\geq \delta \right)=0.
\end{equation}
Therefore, if we set
\[
A_t:=\left\{\left|\frac{S_{T_t}}t-\alpha\right|\leq \delta, \
S_{T_t+1}> t, \ \left|v_{T_t+1}(t-S_{T_t})-\ell\right|\leq \delta \right\}
\]
then, by \eqref{mire} and
by the definition of $K_t$ above, we obtain that
for all $\delta\in\,]0,(1-\alpha)/2[$
\[
\lim_{t\to+\infty} \bbP^{t,\delta}\left(A_t \right)=1.
\]
Moreover on $A_t$ we have $N_t=T_t$ and therefore by \eqref{1} on $A_t$
\begin{equation}\label{ille}
\overline\mu_t(dq,dp)=\frac{dq}t\sum_{i=1}^{T_t} \frac{\delta_{v_i}(dp)}{v_i} 
+\lambda_{v_{T_t+1}(t-S_{T_t})}(dq) 
\, \delta_{v_{T_t+1}}(dp).
\end{equation}
Then for any $f\in C_\mathrm{b}([0,1]\times\R_+)$ we have
\[
\bbP^{t,\delta}(|\overline\mu_t(f)-\mu(f)|>\gep)\leq\bbP^{t,\delta}(\{|\overline\mu_t(f)-\mu(f)|>\gep\}\cap A_t)+
\bbP^{t,\delta}(A_t^c)
\]
and we already know that $\bbP^{t,\delta}(A_t^c)\to 0$
as $t\to+\infty$. Now, by \eqref{ille}, by the law
of the large numbers and by the definition of $A_t$
\[
\lim_{\delta\downarrow 0}\lim_{t\to+\infty} \bbP^{t,\delta}\left(\{|\overline\mu_t(f)-\mu(f)|>\gep\}\cap A_t \right)=0
\]
and we obtain \eqref{sieg}. 

Now we estimate the entropy
\begin{equation}
\label{e:Hbound}
 \begin{split}
 \H({\bf Q}_t \,| \, {\bf P}_t)  & \le 
 \H\left(\bbP^{t,\delta} \, |\, \bbP_{\phi}\right)
= T_t \H(\tilde\pi|\phi) 
-  \log \phi( K_t ),
\end{split}
\end{equation}
so that
\begin{equation*}
 \begin{split}
\lim_{\delta \downarrow 0} \varlimsup_{t\uparrow +\infty}
\frac 1t \H({\bf Q}_t \,| \, {\bf P}_t)  & \le  \mu(p) \H(\tilde\pi \, | \, \phi)
-(1-\alpha) \ell^{-1}  \lim_{\delta \downarrow 0} \varliminf_{\eps \downarrow 0}
\eps \log \phi( [\eps(1-\delta),\eps(1-\delta)[ )
\\ & =  \mu(p) \H(\tilde\pi \, | \, \phi)+(1-\alpha) \ell^{-1}\overline\xi = \overline\I(\mu).
\end{split}
\end{equation*}
Then there exists a map $t\mapsto \delta(t)>0$ vanishing as $t \uparrow +\infty$ such that
${\bf Q}_t:= {\bf Q}_{t,\delta(t)} \to \delta_\mu$ and $\varlimsup_t t^{-1}
\H({\bf Q}_t \,| \, {\bf P}_t) \le \overline\I(\mu)$.

Let now $\mu\in\Omega\setminus\bar\Omega$. Then, by Lemma~\ref{l:Idensity}
we can find a sequence $(\mu_n)_n$ in $\bar{\Omega}$
such that $\mu_n\rightharpoonup\mu$ and $\varlimsup_n \overline\I(\mu_n) \le \overline\I(\mu)$. Moreover, we now know that there exists for all $n$
 a family ${\bf Q}^n_t$ of
probability measures on $\mc P([0,1[\times \R_+)$ such that
${\bf Q}_t^n\rightharpoonup\delta_{\mu_n}$
and
\[
\varlimsup_{t\to+\infty} \frac1t \, \H({\bf Q}_t^n \,| \, {\bf P}_t)\leq \bar\I(\mu_n).
\]
With a standard diagonal procedure we can find a family ${\bf Q}_t$ such that
${\bf Q}_t\rightharpoonup\delta_\mu$
and
\[
\varlimsup_{t\to+\infty} \frac1t \, \H({\bf Q}_t \,| \, {\bf P}_t)\leq \bar\I(\mu).
\]
\end{proof}

\section{Optimality of the bounds}
In this section we show that, for $\xi<\overline\xi$, the law of $\overline\mu_t$
satisfies a large deviations lower bound and a large deviations upper bound with {\it different} rate functionals. 

Let us set 
$\gamma \in  \mc P([0,1[\,\times]0,+\infty[)$ as
\begin{equation*}
\gamma(dq,dp):= 
\begin{cases}
\frac1{\phi(1/p)}\, \frac1p\, \phi(dp)\, \lambda_1(dq)
& \text{if $\phi(1/p)=1/\psi(\tau)\in\, ]0,+\infty[$}
\\
\delta_0(dp)\, \lambda_1(dq) & \text{if $\psi(\tau)=+\infty$}
\end{cases}
\end{equation*}
For $\alpha,\,\ell \in \,]0,1[$ and $\delta >0$, let $\mc A^{\alpha,\ell}_{\delta}$ be the ball of radius $\delta$ in $\mc P(\mc X \times [0,1[)$ centered at $\alpha \gamma + (1-\alpha) \lambda_\ell$ with respect to the standard Prohorov distance. We want to prove that there exist subsequences $(t_k)_k$ and $(s_k)_k$ such that
\[
\lim_{\delta \to 0} \lim_k \frac{1}{t_k} \log \mb P_{t_k} \left( \mc A^{\alpha,\ell}_{\delta}\right) = - (1-\alpha) \ell^{-1} \xi,
\]
\[
\lim_{\delta \to 0} \lim_k \frac{1}{s_k} \log \mb P_{s_k} \left( \mc A^{\alpha,\ell}_{\delta}\right) = - (1-\alpha) \ell^{-1} \bar \xi.
\]
Since the upper and the lower bound are proved, it is enough to prove that there exist subsequences $(t_k)_k$ and $(s_k)_k$ such that
\begin{equation}
\label{e:UBsharp}
\lim_{\delta \to 0} \varliminf_k \frac{1}{t_k} \log \mb P_{t_k} \left( \mc A^{\alpha,\ell}_{\delta}\right) \geq - (1-\alpha) \ell^{-1} \xi,
\end{equation}
\begin{equation}
\label{e:LBsharp}
\lim_{\delta \to 0} \varliminf_k \frac{1}{s_k} \log \mb P_{s_k} \left( \mc A^{\alpha,\ell}_{\delta}\right) \leq - (1-\alpha) \ell^{-1} \bar \xi.
\end{equation}

The inequality \eqref{e:UBsharp} follows in the same way as the lower bound. Take $\mb Q_t$ to be the law of $\overline\mu_t$ when $(\tau_i)_i$ is a sequence of independent variables with law $\psi$ but for $i=\lfloor\alpha t/\psi(\tau)\rfloor$, for which $\tau_i$ has law $\psi (\,\cdot\, | {\tau} \ge t(1-\alpha)/\ell)$.
Then $\H(\mb Q_t \,|\, \mb P_t) \le -\log \psi ([t(1-\alpha)/\ell,+\infty[)$ and
\begin{equation*}
\varliminf_t \frac{1}t \H(\mb Q_t \,|\, \mb P_t)  \le 
\varliminf_t -\frac 1t \log \psi ([t(1-\alpha)/\ell,+\infty[) = \frac{(1-\alpha)}\ell\xi.
\end{equation*}
On the other hand $\overline\mu_t \rightharpoonup \alpha \gamma + (1-\alpha) \lambda_\ell$ under $\mb Q_t$. Therefore the inequality in \eqref{e:UBsharp} is obtained along a subsequence $(t_k)$ realizing the liminf in the above formula.

\medskip
We prove now the inequality \eqref{e:LBsharp}.
Note that for $t>0$,
\begin{equation*}
 \overline\mu_t = \frac{S_{N_t}}t \sum_{i=1}^{N_t} \frac{\tau(x_i)}{S_{N_t}} \delta_{x_i} \otimes \lambda_1 + \frac{t-S_{N_t}}t \delta_{x_{N_t+1}} \otimes \lambda_{\frac{t-S_{\mc N_{t+1}}}{\tau(x_{N_t+1})}}.
\end{equation*}
Let
\[
\begin{split}
\Omega_1:=\Big\{ &
  \mu(dq,dp)= \alpha\, \pi(dp)\, \lambda_1(dq)
   + (1- \alpha)\, \delta_v(dp) \lambda_\ell(dq) :
  \\ &  \alpha\in \,]0,1[, 
  \, \pi \in \mc P(\R_+), \, \pi(p)<+\infty, \, \ell\in [0,1), \, v\geq 0 \Big\}
\end{split}
\]
and notice that we can define a one-to-one map 
\[
\Omega_1\ni\mu\mapsto (\alpha,\pi,\ell,v) \in
\, ]0,1[\, \times \mc P(\R_+)\times [0,1)\times\R_+.
\]
Moreover, if $\Omega_1\ni\mu_n\rightharpoonup\mu\in\Omega_1$, it is easy to
see that necessarily the associated $(\alpha_n,\pi_n,\ell_n,v_n)$ also converge
to $(\alpha,\pi,\ell,v) \in
\, ]0,1[\, \times \mc P(\R_+)\times [0,1)\times\R_+$, and conversely.
Therefore, 
 there exists $\delta_1>0$ such that
\begin{equation*}
 \left\{\overline\mu_t \in \mc A^{\alpha,\ell}_{\delta} \right\} \subseteq 
 \left\{\Big|\frac{S_{N_t}}t - \alpha\Big| \le \delta_1, \   
\Big|\frac{t-S_{N_t}}{\tau_{N_t+1}} - \ell\Big| \le \delta_1
\right\},
\end{equation*}
so that
\[
\mb P_{t} \left( \mc A^{\alpha,\ell}_{\delta}\right) \leq 
\mb P_{t} \left( S_{N_t}\leq t( \alpha+\delta_1), \, \frac{t-S_{N_t}}{\tau_{N_t+1}} \le 
\ell +  \delta_1\right).
\]
Now, let us calculate for $0\le \beta<1$ and $h \in [0,1[$
\begin{equation*}
\begin{split}
& \bb P\left(S_{N_t} \le \beta t,\,\frac{t-S_{N_t}}{\tau_{N_t+1}} \le h  \right) = 
\sum_{n=1}^{\infty} \bb P\left( S_{n} \le \beta t, \tau_{n+1}\ge \frac{t-S_n}h, N_t=n\right)
\\ & \le  \sum_{n=1}^{\infty} \bb P\left(S_{n} \le \beta t, \, \tau_{n+1} \ge \frac{\left(1-\beta\right)}h t\right)
= \sum_{n=1}^{\infty} \bb P\left(S_{n} \le \beta t\right) \, \bb P\left(\tau_{n+1} \ge \frac{\left(1-\beta\right)}h t\right)
\\ & = \psi \left(\left[\frac{\left(1-\beta\right)}h t,+\infty\right[ \right)\, \sum_{n=1}^{\infty} \bb P\left(S_{n} \le \beta t\right)
 =\psi \left(\left[\frac{\left(1-\beta\right)}h t,+\infty\right[ \right)\,  \bb E \left( N_{\beta t} \right).
\end{split}
\end{equation*}
Therefore, recalling that $\bb E (N_t/t)$ stays bounded, by the definition of $\bar\xi$
\begin{equation*}
\begin{split}
& \varlimsup_t \frac{1}{t} \log \bb P\left(S_{N_t} \le \beta t,\,\frac{t-S_{N_t}}{\tau_{N_t+1}} \le h  \right)
\\ &
\quad \le 
\varlimsup_t \frac{1}{t} \log \psi \left(\left[\frac{\left(1-\beta\right)}h t,+\infty\right[ \right)
+\varlimsup_t \frac{1}{t} \log \bb E \left( N_{\beta t} \right)
= - \frac{\left(1-\beta\right)}h \bar \xi
\end{split}
 \end{equation*}
 so that
 \[
 \varlimsup_t \frac{1}{t} \log \mb P_{t} \left( \mc A^{\alpha,\ell}_{\delta}\right) \leq 
 - \frac{\left(1-\alpha\right)}\ell \bar \xi.
 \]
Therefore \eqref{e:LBsharp} is obtained along a subsequence realizing the limsup in the above fomula.

\section{The Donsker-Varadhan rate functional}

In this section we compute the Donsker-Varadhan (DV) functional associated with the Markov process $(q_t,p_t)_{t\geq 0}$. We first check that this process is indeed Markov, we compute its infinitesimal generator and then we compute the associated DV functional.

\subsection{Markov property}
\label{mmpp} 
The following result is intuitively obvious but still requires a proof.
\begin{prop}\label{markov2}
The process $(q_t,p_t)_{t\geq 0}$ is Markov and $(P_t)_{t\geq 0}$
has the semigroup property: $P_{t+s}=P_tP_s$, $t,s\geq 0$.
\end{prop}
\begin{proof}
We define the number of collisions before time $t$
\[
n_t=n_t(q,p,(\tau_n)_{n\geq 1}) := \sum_{n=0}^\infty \un{(T_n\leq t)} = \min\{k\geq 1:
T_k>t \}.
\]
Then we can rewrite
\[
(q_t,p_t) = F(q,p,t,(\tau_n)_{n\geq 1}) :=
\left\{
\begin{array}{ll}
(q_0+p_0t,p_0) \qquad {\rm if} \quad  n_t=0,
\\ \\
\left( \frac{t-T_{n-1}}{\tau_n},\frac 1{\tau_n}
\right)  \quad {\rm if} \quad
 n_t=n\geq 1.
\end{array}
\right.
\]
We claim now that
\begin{equation}\label{iteration}
(q_{t+s},p_{t+s}) = F(q_t,p_t,s,(\tau_{n+n_t})_{n\geq 1}), \qquad \forall \
t,s\geq 0.
\end{equation}
The verification of \eqref{iteration} is a trivial and tedious computation,
where one needs to distinguish the four following
possible situations: $(t+s< T_0)$, $(t< T_0\leq t+s)$, $(T_0\leq t, s<\hat T_0)$ and $(T_0\leq t, \hat T_0\leq s)$. 
We omit the details.

Let us now $\cF_t:=\sigma((q_s,p_s): s\leq t)=\sigma(n_t, \tau_1,
\ldots,\tau_{n_t})$.
Conditionally on $\cF_t$, $(\tau_{n+n_t})_{n\geq 1}$
has the same law as $(\tau_{n})_{n\geq 1}$, since
\[
\begin{split}
& \E(\un{(n_t=k, \, \tau_i\leq x_i, \, i\leq k)} \, f(\tau_{n_t+1}, \tau_{n_t+2},
\ldots)) =
\E(\un{(T_{k-1}\leq t< T_k, \, \tau_i\leq x_i, \, i\leq k)}
\, f(\tau_{k+1},\tau_{k+2},\ldots ))
\\ & = \bbP(T_{k-1}\leq t<T_k, \, \tau_i\leq x_i, \, i\leq k) \,
\E(f(\tau_{k+1}, \ldots))
\\ &  = \bbP(n_t=k, \, \tau_i\leq x_i, \, i\leq k)\, \E(f(\tau_1, \ldots)).
\end{split}
\]
It follows immediately that $(q_t,p_t)$ is a Markov process and
the family of operators $P_tf(q,p):=\E(f(F(q,p,t,(\tau_{n})_{n})))$,
for all bounded Borel function, is a semigroup. Indeed, the conditional
law of $(q_{t+s},p_{t+s})_{s\geq 0}$ given $\cF_t$ is the law
of $F(q_t,p_t,s,(\hat \tau_{n})_{n})$.
\end{proof}

\subsection{The generator}

We want to compute the infinitesimal generator $(L,D(L))$ of the process $(q_t,p_t,t\geq 0)$, in the following weak sense: we say that $f\in D(L)$ if $f:[0,1]\times\R_+\mapsto\R$ is bounded continuous and there exists a bounded continuous $Lf:[0,1]\times\R_+\mapsto\R$ such that
\[
P_tf(q,p)=f(q,p)+ \int_0^t P_s Lf(q,p)\, ds, \qquad \forall \, t\geq 0, \, (q,p)\in
[0,1]\times\R_+.
\]
Let us show that 
\begin{prop}\label{infinitesimal}
The domain $D(L)$ of $L$ is equal to the set of bounded continuous $f:[0,1]\times\R_+\mapsto\R$ such that $(q,p)\mapsto p\frac{\partial f}{\partial q}$ is bounded continuous and
\begin{equation}\label{boucon}
f(1,p) = \int_{\R_+} f(0,1/\tau)\, \psi(d\tau), \qquad \forall \, p\in\R_+,
\end{equation}
and in this case $Lf=p\frac{\partial f}{\partial q}$.
\end{prop}
\begin{proof} We recall that we denote
the law of $\tau_i$ by $\psi(d\tau)$ and the law of $v_i=1/\tau_i$
by $\phi(du)$.  The law of $S_n:=\tau_1+\cdots+\tau_n$ is denoted
as usual by the $n$-fold convolution $\psi^{n*}$ and we recall that
$T_n=T_0+S_n$.
Then we can write
\[
\begin{split}
& P_tf(q_0,p_0)=
\\ & = \un{(t<T_0 )} \, f(q_0+p_0t,p_0)+ \un{(t\geq T_0)} \, \sum_{n=1}^\infty
\E\left(\un{(T_{n-1}\leq t< T_n)} \,
f\left(\frac{t-T_{n-1}}{\tau_n}, \frac1{\tau_n}\right)\right)
\\ & = \un{(t< T_0)} \, f(q_0+p_0t,p_0)+ \\
& \qquad \qquad + \un{(t\geq T_0)} \, \sum_{n=1}^\infty
\int_{[0,t-T_0]} \psi^{*(n-1)}(ds) \int_{]t-T_0-s,+\infty[} \psi(d\tau)\,
f\left(\frac{t-T_0-s}{\tau}, \frac1{\tau}\right)
\\ & = \un{(t< T_0)} f(q_0+p_0t,p_0)+ \un{(t\geq T_0)}
\int_{[0,t-T_0]} U(ds) \int_{]t-T_0-s,+\infty[} \psi(d\tau)\,
f\left(\frac{t-T_0-s}{\tau}, \frac1{\tau}\right)
\end{split}
\]
where we recall that $T_n=T_0(q,p)+\tau_1+\cdots+\tau_n$ and we set
$\psi^{*0}(ds)=\delta_0(ds)$ and
\[
U([a,b]) =\sum_{n=1}^\infty \int_a^b \psi^{*(n-1)}(ds)
= \delta_0([a,b]) + \sum_{n=1}^\infty \int_a^b \psi^{*n}(ds), \quad
0\leq a\leq b.
\]
The {\it renewal measure} $U(ds)$ gives the average number of
collisions in the time interval $ds$. %
Let now $g:[0,1]\times\,]0,+\infty[\, \mapsto\R$ bounded and continuous.
We define 
\[
\begin{split}
I_1(t) & := \int_{[0,1[\times\R_+} dp\, dq\, g(q,p)
\, \un{(t<T_0(q,p))} \, P_tf(q,p) \\ &
= \int_{\R_+} dp \int_0^{(1-tp)^+} dq \, g(q,p) \, f(q+tp,p),
\end{split}
\]
\[
\begin{split}
& I_2(t)  :=\int_{[0,1[\times\R_+} dp\, dq\, g(q,p)
\, \un{(t\geq T_0(q,p))}\, P_tf(q,p)
\\ & = \int_{[0,1[\times\R_+} dp\, dq\, g(q,p) \, \un{(q\geq 1-tp)}
 \int_{[0,t-\frac{1-q}p]} \, U(ds) \cdot
 \\ & \qquad \cdot
\int_{]t-\frac{1-q}p-s,+\infty[} \psi(d\tau)\,
f\left(\frac{t-\frac{1-q}p-s}{\tau}, \frac1{\tau}\right)
\\ & = \int_{\R_+^3} \psi(d\tau) \ U(ds) \ dp
\int_{1\wedge(1+(s-t)p)^+}^{1\wedge(1+(s-t+\tau)p)^+} dq \, g(q,p) \,
f\left(\frac{t-\frac{1-q}p-s}{\tau}, \frac1{\tau}\right).
\end{split}
\]
Let us take the derivative in $t$
\[
\begin{split}
\dot I_1(t) & =
\frac d{dt} \, I_1(t) = \int_0^{\frac 1t} dp \,
p\,\left[\int_0^{1-tp} dq \, g(q,p)\, f_q(q+tp,p)-
g(1-tp,p)\, f(1,p) \right]
\\ & = \int_{[0,1[\times\R_+} dp\, dq\, g
\, \un{(t< T_0)}\, P_tLf\, - \int_{\R_+} dp\, g(1-tp,p)
\, \un{(1-tp\geq 0)}\, f(1,p),
\end{split}
\]
\[
\begin{split}
& \dot I_2(t)=
\frac d{dt} \, I_2(t) =
\\ & = \int_{\R_+^3} \psi(d\tau) \ U(ds) \ dp 
\int_{1\wedge(1+(s-t)p)^+}^{1\wedge(1+(s-t+\tau)p)^+}  dq \,
\frac 1\tau\, g(q,p) \,f_q\left(\frac{t-\frac{1-q}p-s}{\tau}, \frac1{\tau}\right)
\\ & \quad +  \int_{\R_+^3} \psi(d\tau) \ U(ds) \ dp \, p\, \un{(0\leq 1+(s-t)p\leq 1)} \, g(1+(s-t)p,p) \,
f\left(0, {\tau}^{-1}\right)
\\ & \quad -  \int_{\R_+^3} \psi(d\tau) \ U(ds) \ dp \, p\, \un{(0\leq 1+(s-t+\tau)p\leq 1)} \, g(1+(s-t+\tau)p,p) \,
f\left(1, {\tau}^{-1}\right).
\end{split}
\]
Since $\psi* U = U-\delta_0$, if $f$ satisfies the boundary condition \eqref{boucon} above, the last term is equal to
\[
\begin{split}
& - \int_{\R_+^2} U(ds) \ dp \, p\, \un{(0\leq 1+(s-t)p\leq 1)} \, g(1+(s-t)p,p) \,
\int_{\R_+} \psi(d\tau) \, f\left(0, {\tau}^{-1}\right) \\ & +
\int_{\R_+} dp \, p\, \un{(0\leq 1-tp)} \, g(1-tp,p) \,
\int_{\R_+} \psi(d\tau) \, f\left(0, {\tau}^{-1}\right).
\end{split}
\]
By summing all terms, we obtain that for $f$ satisfying \eqref{boucon} 
\[
\begin{split}
& \int _{[0,1[\times\R_+} g
\, \left(P_tf-f\right)  dp\, dq = \int _{[0,1[\times\R_+} g\, \int_0^t P_s Lf\, ds
\,  dp\, dq.
\end{split}
\]
On the other hand, if $f\in D(L)$ then we must have $\dot I_1(t)+\dot I_2(t)\to \int g \, Lf \, dq \, dp$ as $t\to 0$. Now
\[
\dot I_1(0+) = \int_{\R_+} dp \, p\,\left[\int_0^1 dq
\, g(q,p)\, f_q(q,p)- g(1,p)\, f(1,p) \right]
\]
and since $U(ds)=\delta_0(ds)+\un{]0,+\infty[}(s)\, U(ds)$
\[
\dot I_2(0+) = \int_{\R_+} dp \, p\, g(1,p)
\int_{\R_+} \psi(d\tau)\,  f(0,\tau^{-1})
.
\]
Therefore
\[
\begin{split}
& \int_{[0,1[\times\R_+} g \, Lf \, dq \, dp =
\\ & = \int_{\R_+} dp \, \int_0^1 dq \, g(q,p)\, p \,
f_q(q,p) + \int_{\R_+} dp \, p\, g(1,p) \int_{\R_+}
\psi(d\tau)\, (f(0,\tau^{-1})-f(1,p)).
\end{split}
\]
If this is true for all bounded continuous $g$, then $f$ must satisfy \eqref{boucon} above.
\end{proof}

\subsection{Donsker-Varadhan}

 We want to compute,
\[
I(\mu)=\sup_{g\in D(L), g>0}(-\langle g^{-1}L g\rangle_\mu)
\]
in the case of the basic dynamics on the interval $[0,1]$ for which the ``generator" $L$ is given by
\[
Lg(q,p) = p\, \frac{\partial g}{\partial q}, \qquad g(1,\cdot)=
\int_{\R_+} \psi(d\tau)\, g(0,1/\tau).
\]
Assume that $\mu$ is given by a density $f$ wrt Lebesgue on $\R_+$.  We write first,
\[
\langle g^{-1}L g\rangle_\mu=-\int d\mu\, p\, \frac{\partial }{\partial q}\log g.
\]
Then, assuming that 
\[
\sup_{g\in D(L), g>0}(-\langle g^{-1}L g\rangle_\mu)=\sup_{g\in D(L), g>0}
-\int d\mu\, p\, \frac{\partial }{\partial q}\log g<+\infty,
\]
easily implies that $\mu(dq,dp)= dq\, f(dp)$. Then we obtain
\[
\begin{split}
-\langle g^{-1}L g\rangle_\mu=& \int f(dp)\; p \,\log\frac{g(0,p)}{g(1,p)}
\\ =& \int f(dp)\; p \,\log \, g(0,p) - \int f(dp)\; p \, \log\left(\int_{\R_+} \phi(du)\, g(0,u) \right)
\\ =& \mu(p\log g) - \mu(p)\log(\phi(g)) = \mu(p)\left( \tilde\mu(h)-\log(\phi(e^h))\right)
\end{split}
\]
where $g:=e^h$. By Lemma~\ref{algfact}, we can conclude.

\end{document}